\documentclass{amsart}

\newcounter{theoremct}[section]
 \makeatletter
\let\c@equation\c@theoremct  \makeatother

\newtheorem{theorem}[theoremct]{Theorem}
\newtheorem{corollary}[theoremct]{Corollary}
\newtheorem{lemma}[theoremct]{Lemma}
\theoremstyle{remark}
\newtheorem{remark}[theoremct]{Remark}
\theoremstyle{definition}
\newtheorem{definition}[theoremct]{Definition}
\newtheorem*{compcrit}{Composition Condition}
\newtheorem*{invertcrit}{Invertibility Condition}
\newtheorem*{acknowledgment}{Acknowledgement}

\def\L{\relax{L}}
\def\Q{\relax{Q}}
\def\qi{\textup{i}}
\def\qj{\textup{j}}
\def\qk{\textup{k}}
\def\K{\relax{K}}
\def\strut{\vrule width 0pt height 10pt depth 6pt}

\def\Op{\mathfrak{D}}
\def\Ops{\textup{alg}(\Ring,\Op)}
\def\Opsi#1{\textup{alg}_{#1}(\Ring,\Op)}
\def\Ring{\textbf{A}}
\def\rr{a}
\def\Field{\textbf{k}}
\def\End{\textup{End}_{\Field}}
\def\N{\mathbb{N}}
\def\P{\relax{P}}

\def\W{\Phi}
\def\Winv{\Phi^{-1}}

\def\hatOp{\hat\Op}
\begin{document}

\title{On Factoring an Operator Using Elements of its Kernel}

\author{Alex Kasman}

\begin{abstract}
A well-known theorem factors a scalar coefficient differential operator given a linearly independent set of functions in its kernel.  The goal of this paper is to  generalize this useful result to other types of operators.  In place of the derivation $\partial$ acting on some ring of functions, this paper considers the more general situation of an endomorphism $\Op$ acting on a unital associative algebra.  The operators considered, analogous to differential operators, are those which can be written as a finite sum of powers of $\Op$ followed by left multiplication by elements of the algebra.  Assume that the set of such operators is closed under multiplication and that a Wronski-like matrix produced from some finite list of elements of the algebra is invertible (analogous to the linear independence condition).  Then, it is shown that the set of operators whose kernels contain all of those elements is the left ideal generated by an explicitly given operator.  In other words, an operator has those elements in its kernel if and only if it has that generator as a right factor.  Three examples demonstrate the application of this result in different contexts, including one in which $\Op$ is an automorphism of finite order.
\end{abstract}
\maketitle

\section{Introduction}

If $L$ is an ordinary differential operator in the variable $x$ and $f_1(x),\ldots,f_k(x)$ are linearly independent functions in its kernel then $L=Q\circ K$ where $Q$ is a differential operator and $K$ is the differential operator of order $k$ whose action on an arbitrary function $y(x)$ is given by the formula
$$
K(y)=\frac{\hbox{Wr}(f_1,\ldots,f_k,y)}{\hbox{Wr}(f_1,\ldots,f_k)}
$$
with ``$\hbox{Wr}$'' denoting the Wronskian determinant \cite{Zettl}.  This fact of differential algebra has found frequent application, for example, in solving soliton equations by the use of Darboux transformations \cite{Matveev}.

The goal of this note is to extend this useful result regarding factorization to a more general situation.  An ordinary differential operator is a polynomial in the operator $\partial=\frac{d}{dx}$ with coefficients in some ring of differentiable functions on which the operator acts.  To generalize the result, consider the situation of an endomorphism $\Op$ on a unital associative algebra $\Ring$. The main result (Theorem~\ref{mainresult}) is an analogous factorization given elements of the kernel of an operator which can be written as a sum of non-negative integer powers of $\Op$ followed by left multiplication by elements of $\Ring$.  This result requires only mild assumptions on $\Op$ and the given elements of its kernel.  In particular, it is assumed only that the left $\Ring$-module generated by powers of $\Op$ is closed under composition and that a Wronski-like matrix made out of the elements of the kernel is invertible.

Section~\ref{sec:examples} illustrates the use of this result in three situations different from the usual case of differential operators with scalar coefficients.  The first involves differential operators but with non-commuting coefficients, the second considers the case of difference operators acting on discrete scalar functions, and the final example considers an automorphism of order four acting on  the integers with a fifth root of unity adjoined.

\section{Preliminaries}


Let $\Ring$ be an associative unital algebra over the unital commutative ring $\Field$.   The additive and multiplicative identities of $\Ring$ will be denoted by $0$ and $1$ respectively.

Despite the non-commutativity of $\Ring$, we will consider matrices with elements in $\Ring$ and the usual matrix product for which the $(i,j)$ entry of the product of an $m\times n$ matrix $\Omega$ with an $n\times q$ matrix $\Gamma$ is
$$
(\Omega\cdot \Gamma)_{ij}=\sum_{l=1}^n\Omega_{il}\Gamma_{lj}.
$$
We say $\Gamma=\Omega^{-1}$ is the inverse of $\Omega$ if $(\Omega\cdot\Gamma)_{ij}=(\Gamma\cdot\Omega)_{ij}=\delta_{ij}$  where $\delta_{ij}$ is $1\in\Ring$ if $i=j$ and $0\in\Ring$ otherwise.


 For the remainder of the paper, consider a fixed but unspecified endomorphism $\Op\in\End(\Ring)$.   Additionally, we will consider any element $x\in\Ring$ to be an element of $\End(\Ring)$ by identifying it with the endomorphism $L_x:y\mapsto xy$ that left multiplies by $x$.  Since $x\mapsto L_x$ is injective, we will similarly identify $\Ring$ with its isomorphic image in $\End(\Ring)$ under this injection.  Then, together $\Ring\subset\End(\Ring)$ and $\Op\in\End(\Ring)$ generate a subring of operators on $\Ring$ which we will call $\Ops$ where multiplication denoted with the symbol ``$\circ$'' is given by composition.  Note that multiplication in $\Ops$ is then necessarily associative, but probably not commutative.

We will henceforth assume that the endomorphism $\Op$ has the following additional property:

\begin{compcrit}
 For each $f\in\Ring$, the operator $\Op\circ f\in\Ops$  can be written in the form 
$$
\Op\circ f = p_f\circ\Op+q_f
$$
for some $p_f,q_f\in\Ring$. 
\end{compcrit}

 It follows from this assumption that 
  the set $\{\Op^i\ :\ i\in\N\cup\{0\}\}$ spans $\Ops$ as a left $\Ring$-module:

\begin{theorem}\label{thm:allpoly}   For every $\L\in\Ops$ there exist coefficients $\rr_i\in\Ring$ and a number $m\in\N\cup\{0\}$ so that\footnote{For convenience, let $\Op^0$ denote $1\in\Ring$.} $\displaystyle\L=\sum_{i=0}^m \rr_i\circ\Op^i.$
\end{theorem}
\begin{proof}
Let $\L\in\Ops$, then $\L$ can be written as a sum of finite products involving $\Op$ and elements of $\Ring$.  Because of the associative and distributive properties of operator multiplication, it is sufficient to ensure that any product of the form
$$
\rr_1 \circ \Op^m\circ \rr_2\circ \Op^n
$$
can be written in the claimed form where $\rr_1,\rr_2\in\Ring$.

If $m=0$ then this is clearly true since the operator identity $\rr_1\circ 1\circ \rr_2\circ \Op^n=(\rr_1\rr_2)\circ\Op^n$  holds.
Now, suppose it is known that for any $0\leq m<M$ that such an expression is equivalent to a sum of powers of $\Op$ left multiplied by elements of $\Ring$.  Then consider
\begin{eqnarray*}
\rr_1\circ\Op^M\circ \rr_2\circ\Op^n&=&\rr_1\circ\Op^{M-1}\circ (\Op\circ \rr_2)\circ\Op^n\\&=&\rr_1\circ\Op^{M-1}\circ(p_{\rr_2}\circ\Op+q_{\rr_2})\circ \Op^n\\
&=&\rr_1\circ\Op^{M-1}\circ p_{\rr_2}\circ\Op^{n+1}+\rr_1\circ \Op^{M-1}\circ q_{\rr_2}\circ\Op^n.
\end{eqnarray*}
By assumption, each of these terms is of the desired form and hence so is their sum.  Thus, it is true when $m=M$ as well and the claim follows by induction.
\end{proof}

One example to have in mind (called ``the differential operator case'') is where $\Op=\partial$ is the operator that differentiates with respect to the variable $x$ and $\Ring$ is a commutative ring of infinitely differentiable functions of $x$.  In that case, the main result of this paper is a well-known factorization of an ordinary differential operator of order $m>k$ into operators of order $m-k$ and $k$ given a linearly independent set $\{f_1,\ldots,f_k\}$ of functions in its kernel.   

However, the situation described above is much more general.  In fact, because no assumption is made about the independence of the powers of $\Op$, it is not even possible in general to talk about  ``the order'' of  the elements of $\Ops$ as they may have multiple representations as linear combinations of powers of $\Op$.  In place of the familiar notion of the order of a differential operator, we introduce the following filtration.

\begin{definition}
For $n\in\N\cup\{0\}$, let
$$
\Opsi{n}=\left\{\L\in\Ops\ :\ \L=\sum_{i=0}^n\rr_i\circ\Op^i,\ \hbox{for some}\ \rr_i\in\Ring\right\}
$$
denote the left $\Ring$-module generated by the operators $\{\Op^0,\Op^1,\ldots,\Op^n\}$.  \end{definition}

This gives $\Ops$ the structure of an ascending filtration:   
By Theorem~\ref{thm:allpoly}, every $\L\in\Ops$ is in $\Opsi{n}$ for some $n$; merely adding coefficients of like powers of $\Op$ demonstrates that $\Opsi{m}$ is closed under addition; 
the fact that $\Opsi{m}\circ\Opsi{n}\subseteq\Opsi{m+n}$ is a consequence of the Composition Condition; and since it is always possible to add an additional term with a zero coefficient, $\Opsi{n}\subseteq \Opsi{n+1}$.  

\section{An Operator with Specified Kernel}

Now we select a set of elements $\{f_1,\ldots,f_k\}\subset\Ring$.  They may be chosen arbitrarily except for the restriction that a certain matrix built out of them must be invertible:

\begin{invertcrit}
Select and fix $k\in\N$ and $f_1,\ldots,f_k\in\Ring$ with the property that the matrix $\W$ is invertible where
$$
\W=\left(\begin{matrix}
f_1&f_2&\cdots&f_k\cr
\Op (f_1)&\Op (f_2)&\cdots&\Op (f_k)\cr
\Op^2 (f_1)&\Op^2 (f_2)&\cdots&\Op^2 (f_k)\cr
\vdots&\ddots&\ddots&\vdots\cr
\Op^{k-1} (f_1)&\Op^{k-1} (f_2)&\cdots&\Op^{k-1} (f_k)
\end{matrix}\right).
$$

\end{invertcrit}

\begin{remark}
For some algebras $\Ring$ and endomorphisms $\Op$, there may be very few choices of elements $f_1,\ldots,f_k$ for which $\W$ is invertible.  Theorem~\ref{mainresult} will be of limited value in such cases.  In fact, if the operator identity
$$
0=\sum_{i=0}^m \rr_i \circ\Op^i
$$
holds for some $\rr_i\in\Ring$ with $\rr_m\not=0$ then no such set with $k>m$ satisfies the invertibility condition.  Note, however, that there is always at least one such choice since $k=1$ and $f_1=1\in\Ring$ always satisfies the condition.

\end{remark}
\begin{remark}
Clearly, the matrix $\W$ is an analogue of the matrix in the differential operator case whose determinant is the Wronskian.  Generally, its chief significance is that a non-zero Wronskian indicates that the set $\{f_1,\ldots,f_k\}$ of scalar functions is linearly independent.  However, as far as this note is concerned the significance of the matrix $\W$ is that its invertibility allows us to construct operators in $\Opsi{k-1}$ that are ``dual'' to the elements $f_i$ as shown in the following lemma.
\end{remark}

\begin{lemma}\label{lem:P_i}
The operator $$\P_i=\sum_{l=1}^{k}(\Winv)_{il}\circ\Op^{l-1}$$ satisfies $\P_i(f_j)=\delta_{ij}$ for each $1\leq i,j\leq k$.
\end{lemma}
\begin{proof}
The claim follows from the assumption that $\Winv$ is the inverse of $\W$ because
$$
\P_i(f_j)=\sum_{l=1}^{k}\left((\Winv)_{il}\circ\Op^{l-1}\right)(f_j)=\sum_{l=1}^{k}(\Winv)_{il}\left(\Op^{l-1}f_j\right)=(\Winv\cdot\W)_{ij}.
$$
\end{proof}

As a consequence, we can produce an operator $\hat \P\in\Opsi{k-1}$ which does anything we wish to the elements $f_1,\ldots, f_k$:
\begin{corollary}\label{hatP}
For any $\hat f_i\in\Ring$ ($1\leq i\leq k$) let $\hat \P=\sum \hat f_i\P_i$.  Then
$
\hat \P(f_i)=\hat f_i.
$
\end{corollary}

In particular, we may build out of the operators $\P_i$ an operator in $\Opsi{k}$ having each of the elements $f_i$ ($1\leq i\leq k$) in its kernel:
\begin{definition}\label{def:K}
Given $f_i\in\Ring$ satisfying the Invertibility Condition,
let $\hat f_i=\Op^k (f_i)$ and let $\hat \P$ be the corresponding operator from Corollary~\ref{hatP}.  Finally, define $\K$ to be
$\displaystyle\K=\Op^k-\hat \P$. 
\end{definition}
\begin{theorem}\label{thm:K}
 The kernel of operator $\K$ from Definition~\ref{def:K} contains each of the elements $f_i$: $\K (f_i)=0$ for $1\leq i\leq k$.
\end{theorem}
\begin{proof}
\begin{eqnarray*}
\K( f_i)&=&(\Op^k-\hat \P)(f_i)
=\Op^k (f_i )-\hat f_i\hbox{ (by Corollary~\ref{hatP})}\\
&=&\Op^k(f_i)-\Op^k(f_i)\hbox{ (by definition of $\hat f_i$)}\\
&=&0.
\end{eqnarray*}
\end{proof}

\begin{remark}
In the scalar differential operator case, $\K$ is the order $k$ differential operator whose action on an arbitrary function is the same as the Wronskian determinant  $\hbox{Wr}(f_1,\ldots,f_k,y)$ divided by the Wronskian determinant  $\hbox{Wr}(f_1,\ldots,f_k)$.  More generally, when $\Ring$ is a non-commutative ring of functions and $\Op=\partial$ differentiates those functions with respect to the variable, then this is the same as the operator produced using quasideterminants in \cite{QD}.
\end{remark}

\section{Another Spanning Set for $\Ops$}

As previously noted, every element of $\Ops$ can be written as a finite linear combination of terms of the form $\Op^i$ ($i\geq 0$) left multiplied by elements of $\Ring$.  This section will introduce another set spanning $\Ops$ as a left $\Ring$-module which depends on the choice of elements $f_1,\ldots,f_k\in\Ring$ satisfying the Invertibility Condition.

\begin{definition}
For $i\in\N\cup\{0\}$ define $\hatOp_i$ by
$$
\hatOp_i=\left\{\begin{matrix} \P_{i+1}&\hbox{if $0\leq i\leq k-1$}\\
\\
\Op^{i-k}\circ \K&\hbox{if $i\geq k$}\end{matrix}\right.
$$
where $\P_i$ and $\K$ are the operators from Lemma~\ref{lem:P_i} and Definition~\ref{def:K}.  
\end{definition}

\begin{lemma}\label{lem:coefs}
For any $\L\in\Opsi{m}$ with $m\geq k-1$, there exist elements $\hat \rr_i\in\Ring$ such that
$$
\L=\sum_{i=0}^m\hat \rr_i \circ\hatOp_i.
$$
\end{lemma}
\begin{proof}
Let $\L=\sum_{i=0}^m\rr_i\circ\Op^i$.

First, consider the case that $m=k-1$.
Define $\hat \rr_{i-1}$ for $1\leq i\leq k$ to be the matrix product of $(\rr_0\ \cdots\ \rr_{k-1})$ with the $i^{th}$ column of $\W$ and $\hat \rr_i=0$ for $i\geq k$.
The claim then follows because for any $f\in\Ring$ \begin{eqnarray*}
\L(f)&=&(\rr_0\ \cdots\ \rr_{k-1})\cdot\left(\begin{matrix}f\\\Op f\\\Op^2 f\\\vdots\\\Op^{k-1} f\end{matrix}\right)
=(\rr_0\ \cdots\ \rr_{k-1})\cdot \W\cdot \Winv\cdot \left(\begin{matrix}f\\\Op f\\\Op^2 f\\\vdots\\\Op^{k-1} f\end{matrix}\right)\\
&=&(\rr_0\ \cdots\ \rr_{k-1})\cdot \W\cdot \left(\begin{matrix}\hatOp_0 f\\\hatOp_1 f\\\hatOp_2 f\\\vdots\\\hatOp_{k-1} f\end{matrix}\right)
=(\hat \rr_0\ \cdots\ \hat \rr_{k-1})\cdot \left(\begin{matrix}\hatOp_0 f\\\hatOp_1 f\\\hatOp_2 f\\\vdots\\\hatOp_{k-1}f \end{matrix}\right).
\end{eqnarray*}

Now, we proceed by induction on $m$ by assuming that the claim is true if $\L\in\Opsi{m-1}$ for some $m\geq k$.


Define $c_i\in\Ring$ ($0\leq i\leq m-1$) by 
$$
\hatOp_m=\Op^{m-k}\circ(\Op^k-\hat P)=\Op^m-\Op^{m-k}\circ \hat P=\Op^m+\sum_{i=0}^{m-1}c_i\Op^i
$$
and let $\L_-=\L-\rr_m\circ \hatOp_m$.  Then, because the coefficients of the degree $m$ term cancel, 
$$
\L_-=\sum_{i=0}^{m-1}(\rr_i-c_i)\circ \Op^i.
$$
Hence, $L_-\in\Opsi{m-1}$.  By the induction hypothesis, $\L_-$ can be written as a linear combination of the terms $\hatOp_i$ with $0\leq i\leq m-1$ having coefficients $\hat \rr_i\in\Ring$.  But, by definition, $\L$ is obtained by simply adding $\rr_m\circ\hatOp_m$ to this.
\end{proof}

If one attempted to turn the preceding inductive proof of the existence of the coefficients $\hat \rr_i$ into an algorithm for obtaining them, it would involve finding the coefficients with indices $i\geq k$ before computing the lower ones.  However, once we know that they exist, the first $k$ of them can be determined more simply by applying the desired operator to the chosen elements $f_i$:
\begin{theorem}\label{thm:key}
If
$\displaystyle
\L=\sum_{i=0}^m\hat \rr_i\circ\hatOp_i,
$
then
$
\hat \rr_{i-1}=\L(f_{i})$ for $1\leq i\leq k$.
\end{theorem}
\begin{proof}  Let $1\leq i\leq k$.
\begin{eqnarray*}
\L(f_{i})&=&\sum_{j=0}^m \hat \rr_j \circ\hatOp_j(f_i)\\
&=& \sum_{j=0}^{k-1}\hat \rr_j\circ\hatOp_j(f_i)+\sum_{j=k}^{m}\hat \rr_j\circ\hatOp_j(f_i)\\
&=& \sum_{j=0}^{k-1}\hat \rr_j(\hat\P_{j+1}(f_i))+\sum_{j=k}^{m}\hat \rr_j(\Op^{j-k}\circ K(f_i))\\
&=&\hat \rr_{i-1}+\sum_{j=k}^{m}\hat \rr_j(\Op^{j-k}\circ K(f_i)) \hbox{ (by Lemma~\ref{lem:P_i})}\\
&=&\hat \rr_{i-1}+0=\hat \rr_{i-1}\hbox{ (because $f_i\in\ker K$)}.
\end{eqnarray*}
\end{proof}

Even before we prove the factorization theorem in the next section, Theorem~\ref{thm:key} rules out the possibility of a non-zero element of $\Opsi{k-1}$ having each $f_i$ in its kernel:

\begin{corollary}If $\L\in\Opsi{k-1}$ and $L(f_i)=0$ for $1\leq i\leq k$ then $\L=0$ is the zero operator.
\end{corollary}
\begin{proof}
By Lemma~\ref{lem:coefs},  if $\L\in\Opsi{k-1}$  then $\L$ can be written as a linear combination of the operators $\hatOp_i$ for $0\leq i\leq k-1$.  But, by  Theorem~\ref{thm:key}, the coefficients are all obtained by applying $\L$ to the elements $f_i$.  So, if those elements are in the kernel of $\L$ then it is equal to a linear combination with all zero coefficients and hence is zero itself.
\end{proof}

\section{Factorization}

The main result is that an operator $\L\in\Ops$ has each of the elements $f_i$  in its kernel precisely when it has a right factor of the operator $\K$ from Definition~\ref{def:K}.

\begin{theorem}\label{mainresult}
For $\L\in\Ops$, $\L(f_i)=0$ for $1\leq i\leq k$ if and only if there exists $\Q\in\Ops$ such that $\L=\Q\circ \K$.
\end{theorem}
\begin{proof}  
Clearly, if $\L=\Q\circ\K$ for some $\Q\in\Ops$ then $\L(f_i)=0$ for each $1\leq i\leq k$ because $\K(f_i)=0$ and multiplication in $\Ops$ is defined to coincide with operator composition.

Now, suppose $\L(f_i)=0$ for $1\leq i\leq k$ and 
let $m\in\N$ be large enough that $m>k$ and $\L\in\Opsi{m}$.  Then by Lemma~\ref{lem:coefs} there exist coefficients $\hat \rr_i$ such that
$\displaystyle
\L=\sum_{i=0}^m\hat \rr_i\circ\hatOp_i$.
Using Theorem~\ref{thm:key} to compute the first $k$ coefficients we determine that $\hat \rr_i=\L(f_{i+1})=0$ for $0\leq i\leq k-1$(by the assumption that $f_i\in\ker \L$).  Eliminating these terms with known zero coefficients from the sum we get
$$
\L=\sum_{i=0}^m\hat \rr_i\circ\hatOp_i=\sum_{i=k}^m\hat \rr_i\circ\hatOp_i.
$$
But since $\hatOp_i=\Op^{i-k}\circ K$ for $i\geq k$ this can be written as
$$
\L=\sum_{i=k}^m\hat \rr_i\circ\Op^{i-k}\circ K=\left(\sum_{i=0}^{m-k}\hat \rr_{i+k}\circ\Op^{i}\right)\circ K.
$$
Then the claim is true with  $\Q$ being defined as the left factor in the last product.
\end{proof}

\section{Examples}\label{sec:examples}

\subsection{Differential Operators with rational quaternionic coefficients}

Let $\Field=\mathcal{Q}_8$ be the quaternions.  An element of $\Field$ is of the form $a+b\qi + c\qj+ d \qk$ where $a$, $b$, $c$ and $d$ are real numbers and $\qi$, $\qj$ and $\qk$ are non-commuting elements satisfying $\qi^2=\qj^2=\qk^2=\qi\qj\qk=-1$. Let $\Ring=\Field(x)$ be the ring of rational functions of the real variable $x$ with quaternionic coefficients.  

The differential operator $\Op=\frac{d}{dx}$ satisfies the Composition Condition because $\Op\circ f(x)=f(x)\circ\Op+f'(x)$ is a consequence of the product rule.

With $k=2$, $f_1=x\qk$ and $f_2=x^3\qi$ we get that the Wronskian matrix $\Phi$ has inverse $\Phi^{-1}$ where
$$
\Phi=\left(\begin{matrix}
x\qk&x^3\qi\strut\cr
\qk&3x^2\qi\strut\end{matrix}\right)
\qquad
\hbox{and}
\qquad
\Phi^{-1}=\left(\begin{matrix}
\frac{-3}{2x}\qk&\frac{1}{2}\qk\strut\cr
\frac{1}{2x^3}\qi&\frac{-1}{2x^2}\qi\strut\end{matrix}\right).
$$

Then, Lemma~\ref{lem:P_i} tells us that the operators
$$
P_1=\frac{-3}{2x}\qk+\frac{1}{2}\qk\circ\Op
\qquad
\hbox{and}
\qquad
P_2=\frac{1}{2x^3}\qi+\frac{-1}{2x^2}\qi\circ\Op
$$
(whose coefficients come from the first and second rows of $\Phi^{-1}$ respectively) satisfy $P_i(f_j)=\delta_{ij}$.  This can easily be verified by computation.

Moreover, letting $\K=\Op^2-\hat f_1 \circ P_1-\hat f_2 \circ P_2$ where $\hat f_i=\Op^2(f_i)$ we get
$$
\K=\Op^2-\frac{3}{x}\circ\Op+\frac{3}{x^2}.
$$
Indeed, $\K(f_1)=\K(f_2)=0$.  
Another element of $\Ops$ having both $f_1$ and $f_2$ in its kernel is $\L=x^3\qj\Op^3+(x^2\qi-3x^3\qj)\circ\Op^2-3(x\qi-2x\qj)\circ\Op+3(\qi-2\qj)$. Using
Theorem~\ref{mainresult} we know immediately that $\L$ has a right factor of $K$.

\subsection{Difference Operators}

Let $\Op$ be the difference operator acting on the ring $\Ring$ of real rational functions of the discrete variable $n$ by the formula $\Op(g(n))=g(n+1)+ c g(n)$ where $c$ is some constant.  Then $\Op$ satisfies the Composition Condition since for any function $f(n)$
\begin{eqnarray*}
(\Op\circ f)g(n)&=&\Op(f(n)g(n))=f(n+1)g(n+1)+cf(n)g(n)\\
&=&
[f(n+1)\Op+cf(n)-cf(n+1)](g(n)).
\end{eqnarray*}
Then for any functions $f_1(n),\ldots,f_k(n)$ for which the corresponding matrix $\Phi$ is invertible, Theorem~\ref{thm:K} provides 
a method for producing an operator $K$ that is polynomial in $\Op$ with coefficients that are functions of $n$ having each of those functions in its kernel, and Theorem~\ref{mainresult} ensures that any such operator with those functions in its kernel has a right factor of $K$.

For example, if $f_1=n$ and $f_2=n^2$ then the Casoratian matrix $\Phi$ and its inverse would be
$$
\Phi=\left(\begin{matrix}
n&n^2\strut\cr
(c+1)n+1&\strut(c+1)n^2+2n+1
\end{matrix}\right)
\ 
\hbox{and}
\ 
\Phi^{-1}=\left(\begin{matrix}
 \frac{(c+1) n^2+2 n+1}{n^2+n} & -\frac{n}{n+1}\strut \cr
 -\frac{c n+n+1}{n^2+n} & \frac{1}{n+1} \strut
\end{matrix}\right).
$$
The rows of $\Phi^{-1}$ give us coefficients for operators $P_i$ in $\Ops$ with the property that $P_i(f_j)=\delta_{ij}$ and out of them we can build the operator 
$$
\K=\Op^2-\frac{2 (c n+c+n+2)}{n+1}\circ\Op+\frac{\left(c^2+4 c+3\right) n+(c+1)^2 n^2+2}{n (n+1)}
$$
from Definition~\ref{def:K} which has $n$ and $n^2$ in its kernel.  In fact, Theorem~\ref{mainresult} assures us that a difference operator $\L\in\Ops$ has these two functions in its kernel if and only if $\L=\Q\circ\K$ for some $\Q\in\Ops$.

\subsection{Fifth Roots of Unity}

For a very different example, consider the algebra $\Ring=\mathbb{Z}[\rho]$  over the ring $\Field=\mathbb{Z}$ of integers, where $\rho^5=1\not=\rho$.  Any element of $x\in\Ring$ then has a unique representation of the form $x=a+b\rho+c\rho^2+d\rho^3+e\rho^4$ where $a,b,c,d,e\in\mathbb{Z}$.

The automorphism that permutes the fifth roots of unity:
\begin{eqnarray*}
\Op:\Ring&\to&\Ring\\
a+b\rho+c\rho^2+d\rho^3+e\rho^4&\mapsto&
a+d\rho+b\rho^2+e\rho^3+c\rho^4
\end{eqnarray*}
distributes linearly over $\Field$ linear combinations.  However, unlike the previous examples this endomorphism is unipotent since $\Op^4=\Op^0$ is the identity map.  Consequently, $\Ops=\Opsi3$.

Nevertheless, the results of this note apply to this situation 
since the automorphism satisfies the Composition Condition.  In fact, since it is also a group homomorphism, for any $x\in\Ring$ we have  simply
$
\Op\circ x=\Op(x)\circ\Op$.

If $k=1$ and $f_1=\rho^2$ then Definition~\ref{def:K} allows us to produce an operator $\K\in\Ops$ having $f_1$ in its kernel.  In particular, since $\Phi^{-1}=\rho^3=P_1$ we have simply that
$$
\K=\Op-[\Op(f_1)]\circ P_1=\Op-\rho^4\rho^3=\Op-\rho^2.
$$
Indeed, $\K(\rho^2)=\rho^4-\rho^4=0$.  More interestingly, we know that an operator in $\Ops$ 
satisfies $\L(\rho^2)=0$ if and only if $\L=\Q\circ\K$ for some $\Q\in\Ops$.  In particular, even though it is not immediately obvious that $\L=\rho\Op^3-1$ is a multiple of $\Op-\rho^2$, since $\L(\rho^2)=0$ we know from Theorem~\ref{mainresult} that it is.  (Indeed,  one can check that $\L=(\rho\circ \Op^2+\rho^4\circ\Op+\rho^3)\circ K$.)

\section{Concluding Remarks}

\begin{remark}
Another way to state the main result (Theorem~\ref{mainresult}) would be to say that given the Composition and Invertibility Conditions, the set of elements in $\Ops$ having the elements $f_i$ ($1\leq i\leq k$) in their kernel is precisely the left ideal generated by $K$.
\end{remark}
\begin{remark}
As already noted, in the case that $\Op=\frac{d}{dx}$ and $\Ring$ is a (non-commutative) ring of differentiable functions of $x$, the operator $K$ having a specified kernel produced in Theorem~\ref{thm:K} is the same as the operator with this property produced using quasideterminants in \cite{QD}.   Moreover, that paper provides a complete factorization of the operators produced using that procedure.  However, it does not follow from those results that \textit{any} differential operator with coefficients in $\Ring$ and having those functions in the kernel must have that operator $K$ as a right factor.  Theorem~\ref{mainresult} above shows that this is the case, even if that operator has a non-invertible leading coefficient and hence could not be one of the operators produced by the quasideterminant procedure.
\end{remark}

\begin{remark}
For other results about the factorization of differential operators given functions in their kernels, see \cite{MatrixFactNote} which factors matrix coefficient ordinary differential operators given vector functions in their kernel and \cite{FactorPDOs} which factors constant coefficient scalar partial differential operators whose kernels contain certain continuous families of functions.
\end{remark}

\begin{remark}An application of these results is the ability to produce intertwining relationships for operators:
\begin{corollary}
Suppose $\{f_1,\ldots,f_k\}\subset\Ring$ and $\K\in\Ops$ are as in Theorem~\ref{mainresult} and that $R\in\Ops$ is an operator with the property that $R(f_i)\in\ker\K$ for $1\leq i\leq k$.  
Then $\K\circ R=\Q\circ K$ for some operator $\Q\in\Ops$.  
\end{corollary}
\begin{proof}
Since $\K\circ R(f_i)=\L(R(f_i))$, if $R(f_i)$ is in the kernel of $\K$ for each $1\leq i\leq k$ then $\K\circ R(f_i)=0$.  According to Theorem~\ref{mainresult}, there exists an operator $\Q\in\Ops$ such that $\K\circ R=\Q\circ \K$.
\end{proof}
In the differential operator case, such intertwining relationships are useful for constructing Darboux transformations in which a new operator $\Q$ having some desired property is produced from a known operator $R$ having that property \cite{Matveev}.
\end{remark}

\begin{remark} 
In some cases, the operator $K$ defined in Definition~\ref{def:K} turns out to be the zero operator.  Still, Theorem~\ref{mainresult} applies and one concludes that the zero operator is the only element of $\Ops$ having each $f_i$ in its kernel.
\end{remark}

\begin{remark}
The endomorphisms $\Op$ being considered in this paper are closely related to ``skew derivations'' \cite{skew}.   After reading an early draft of this note, Oleg Smirnov noted that an endomorphism satisfies the Composition Condition if and only if it can be written as the sum of a skew derivation and a left multiplication operator.  Hence, any skew derivations could be used to provide additional examples of endomorphisms for which the main result applies.  Note, however, that only the first of the three examples given in Section~\ref{sec:examples} is a skew derivation.
\end{remark}
 \begin{acknowledgment}
 The author wishes to thank Maarten Bergvelt, Michael Gekhtman, Tom Kunkle,
 Chunxia Li and Oleg Smirnov for advice, assistance and encouragement.
 \end{acknowledgment}

\end{document}